\documentclass[10pt]{amsart}
\usepackage{enumerate,amsmath,amssymb,latexsym,
amsfonts, amsthm, amscd, mathabx}

\theoremstyle{plain}

\newtheorem{theorem}{Theorem}

\numberwithin{equation}{section}

\newcommand{\ra}{\rightarrow}

\newcommand{\FF}{\mathcal{F} (V)}

\addtocounter{section}{-1}

\begin{document}

\title {The Clifford algebra and its antidual}

\date{}

\author[P.L. Robinson]{P.L. Robinson}

\address{Department of Mathematics \\ University of Florida \\ Gainesville FL 32611  USA }

\email[]{paulr@ufl.edu}

 \keywords{}

\begin{abstract}

We analyze the purely algebraic antidual $C'(V)$ of the complex Clifford algebra $C(V)$ over a real inner product space $V$. In particular, we introduce a partially defined product in $C'(V)$ and study its properties.  

\end{abstract}

\maketitle

\medbreak 

\section*{The Antidual}

\medbreak 

Throughout, $V$ will be an infinite-dimensional real vector space upon which $( \bullet | \bullet )$ is an inner product. Let $C(V)$ be `the' complex Clifford algebra of this real inner product space: this unital involutive associative complex algebra is generated by its real subspace $V$  whose elements $v$ satisfy $v^* = v$ and $v^2 = ( v | v ) \bf{1}$. 

\medbreak 

The Clifford algebra $C(V)$ carries an array of natural structures. In addition to its involution, $C(V)$ admits a unique automorphism $\gamma : C(V) \ra C(V)$ such that $\gamma (v) = - v$ whenever $v \in V$: this grading automorphism induces a direct eigendecomposition $$C(V) = C_+ (V) \oplus C_- (V)$$ where $\gamma |_{C_{\pm} (V)} = \pm I$; accordingly, $C(V)$ is naturally a superalgebra. The involutive algebra $C(V)$ admits a unique even tracial state: that is, a complex-linear functional $\tau : C(V) \ra \mathbb{C}$ normalized by $\tau ({\bf 1}) = 1$ such that if $a, b, c \in C(V)$ then $\tau (ab) = \tau (b a)$ along with $\tau (c^*) = \overline{\tau(c)}$ and $\tau (\gamma (c)) = \tau (c)$. The involution $^*$ and trace $\tau$ together provide $C(V)$ with a canonical $\gamma$-invariant complex inner product $ \langle \bullet | \bullet \rangle$ defined by the rule that if $a, b \in C(V)$ then 
$$\langle a | b \rangle = \tau (a^* b)$$
where we adopt the convention according to which $\langle a | b \rangle$ is antilinear in $a$ and linear in $b$; it is readily verified that if also $c \in C(V)$ then $ \langle a^* c | b \rangle = \langle c | a b \rangle = \langle c b^* | a \rangle$ so that in particular $ \langle a^* | b \rangle = \langle b^* | a \rangle$. For all of this and more, see [1] and [2]. 

\medbreak 

Now, let $C'(V)$ denote the full antidual of $C(V)$: thus, $C'(V)$ comprises precisely all antilinear functionals $C(V) \ra \mathbb{C}$ whether bounded or unbounded. Our primary concern in this paper is to investigate the natural structural features of this antidual. 

\medbreak 

The antidual $C'(V)$ inherits from $C(V)$ its own involution: explicitly, for $\phi \in C'(V)$ the adjoint $\phi^* \in C'(V)$ is defined by the rule that if $a \in C(V)$ then 
$$\phi^* (a) = \overline{\phi (a^*)}.$$
Further, $C'(V)$ inherits its own grading automorphism $\gamma$ defined by the rule that if $\phi \in C'(V)$ and $a \in C(V)$ then 
$$\gamma(\phi) (a) = \phi (\gamma (a)).$$
The eigenspaces of this grading have the property that $C'_{\pm} (V)$ is precisely the annihilator of $C_{\mp} (V)$ as is customary for the dual of a superspace. 

\medbreak 

The antidual $C'(V)$ is also naturally a $C(V)$-bimodule: when $\phi, \psi \in C'(V)$ and $a, b, c \in C(V)$ we define 
$$(a \cdot \psi) (c) = \psi (a^* c) \; \; {\rm and} \; \; (\phi \cdot b) (c) = \phi (c b^*).$$
This is indeed a bimodule structure: it makes $C'(V)$ both a left $C(V)$-module and a right $C(V)$-module, while if $\chi \in C'(V)$ and $a, b \in C(V)$ then 
$$a \cdot (\chi \cdot b) = (a \cdot \chi) \cdot b$$
for if also $c \in C(V)$ then 
$$[ a \cdot (\chi \cdot b)] (c) = (\chi \cdot b) (a^* c) = \chi (a^* c b^*) = (a \cdot \chi) (c b^*) = [(a \cdot \chi) \cdot b] (c)$$
by associativity in $C(V)$. This bimodule structure is compatible with the involution and the grading: if $\phi, \psi \in C'(V)$ and $a, b \in C(V)$ then $(a \cdot \psi)^* = \psi^* \cdot a^*$ and $(\phi \cdot b)^* = b^* \cdot \phi^*$ along with $\gamma (a \cdot \psi) = \gamma(a) \cdot \gamma (\psi)$ and $\gamma (\phi \cdot b ) = \gamma (\phi) \cdot \gamma (b)$. 

\medbreak 

Now, the inner product $\langle \bullet | \bullet \rangle$ on $C(V)$ engenders a canonical complex-linear embedding 
$$\Theta : C(V) \ra C'(V) : a \mapsto \Theta_a$$
given by the rule that if also $b \in C(V)$ then 
$$\Theta_a (b) = \langle b | a \rangle.$$
This embedding respects both the involution and the grading: if $a \in C(V)$ then 
$$\Theta_{a^*} = (\Theta_a)^* \; \; {\rm and} \; \; \Theta_{\gamma(a)} = \gamma( \Theta_a)$$
for if also $b \in C(V)$ then 
$$\Theta_{a^*} (b) = \langle b | a^* \rangle = \overline{ \langle b^* | a \rangle} = \overline{\Theta_a (b^*)} = (\Theta_a)^* (b)$$
and 
$$\Theta_{\gamma(a)} (b) = \langle b | \gamma(a) \rangle = \langle \gamma(b) | a \rangle = \Theta_a (\gamma(b) ) = \gamma(\Theta_a) (b).$$
This embedding also respects the $C(V)$-bimodule stucture: a similarly routine check (which already appears above) shows that if $a, b \in C(V)$ then 
$$a \cdot \Theta_b = \Theta_{a b} = \Theta_a \cdot b.$$

\medbreak 

The antidual $C'(V)$ also supports a canonical extension of the trace $\tau$: in fact, if $a \in C(V)$ then 
$$\tau (a) = \tau ({\bf 1}^* a ) = \langle {\bf 1} | a \rangle = \Theta_a ({\bf 1})$$
so that evaluation at the multiplicative identity ${\bf 1} \in C(V)$ defines an extension of the trace, thus 
$$\tau' : C'(V) \ra \mathbb{C} : \phi \mapsto \phi ({\bf 1})$$
which we shall henceforth denote simply by $\tau$. This extended trace continues to respect both involution and grading: if $\phi \in C'(V)$ then $\tau (\phi^*) = \overline{ \tau (\phi)}$ and $\tau (\gamma (\phi)) = \tau (\phi)$; in short, $\tau \circ \: ^* = \bar{\tau}$ and $\tau \circ \gamma = \tau$. 

\medbreak 

The question whether $\tau$ continues to be a trace in the sense $\tau (\phi \psi) = \tau (\psi \phi)$ demands that we first address the question whether $C'(V)$ is equipped with a canonical product. We should like such a product to make the canonical embedding $\Theta : C(V) \ra C'(V)$ not only linear but also multiplicative and hence an algebra map. Such a product should perhaps also be separately continuous for the weak topology of pointwise convergence: in terms of nets, if $\phi_{\delta} \ra \phi$ and $\psi_{\delta} \ra \psi$ then $(\phi_{\delta} \psi) (c) \ra (\phi \psi) (c)$ and $(\phi \psi_{\delta}) (c) \ra (\phi \psi) (c)$ for each $c \in C(V)$. As we shall see, no such product is defined on the whole of $C'(V)$: we shall be forced to settle for a canonical product that is only partially defined; but there does indeed exist a canonical partial product, which respects the various natural structures presented thus far. We shall set up this partial product and further analyze the antidual in the following section; as preparation for this, we introduce  finite-dimensional approximants of antifunctionals. 

\medbreak 

Direct by inclusion the set $\FF$ comprising all finite-dimensional real subspaces of $V$. Each $M \in \FF$ generates $C(M)$ as a subalgebra of $C(V)$; otherwise put, the inclusion of $M$ in $V$ is isometric and so induces a canonical embedding of $C(M)$ in $C(V)$. Of course, $C(V)$ is the union of its subalgebras $C(M)$ as $M$ runs over $\FF$. In fact, more is true. 

\medbreak 

\begin{theorem} \label{M_a} 
If $a \in C(V)$ then there exists a unique smallest $M_a \in \FF$ such that $a \in C(M_a)$. 
\end{theorem} 

\begin{proof} 
Let $\mathcal{F}_a(V)$ be the collection of all $M \in \FF$ such that $a \in C(M)$. This collection being nonempty, we define $M_a$ to be its intersection:  
$$M_a = \bigcap \mathcal{F}_a(V) \in \mathcal{F}_a(V).$$
Theorem 1.3 of [2] justifies the middle step of the conclusion 
$$a \in \bigcap \{ C(M) : M \in \mathcal{F}_a(V) \} = C( \bigcap \{ M : M \in \mathcal{F}_a(V) \}) = C(M_a).$$
\end{proof} 

\medbreak 

Now, let $\phi \in C'(V)$. To each $M \in \FF$ there is associated a unique $\phi_M \in C(M)$ such that 
$$\phi |_{C(M)} = \langle \cdot | \phi_M \rangle.$$
We claim that these finite-dimensional approximants converge weakly to $\phi$ as $M$ runs through the directed set $\FF$. 

\medbreak 

\begin{theorem} \label{net} 
If $\phi \in C'(V)$ then the net $( \Theta_{\phi_M} : M \in \FF )$ is pointwise convergent  to $\phi$. 

\end{theorem} 

\begin{proof} 
Let $a \in C(V)$: as soon as $M \in \FF$ contains $M_a$ it follows that 
$$\Theta_{\phi_M} (a) = \langle a | \phi_M \rangle = \phi (a);$$
the net $( \Theta_{\phi_M} (a) : M \in \FF )$ is thus eventually constant with value $\phi (a)$ and so convergent to $\phi (a)$ as claimed. 
\end{proof} 

\medbreak 

We may express this result as  
$$\phi = \lim_{M \uparrow \FF} \Theta_{\phi_M}$$
though for convenience we may abuse notation and write simply $\phi_M \ra \phi$ as $M \uparrow \FF$. 

\medbreak 

\begin{theorem} \label{prop}
If $ \phi \in C'(V)$ and $M \in \FF$ then $(\phi^*)_M = (\phi_M)^*$ and $\gamma(\phi)_M = \gamma (\phi_M)$. 
\end{theorem} 

\begin{proof} 
Direct calculation. Let $c \in C(M)$: on the one hand $c^* \in C(M)$ so 
$$\langle c | (\phi^*)_M \rangle = \phi^* (c) = \overline{ \phi (c^*)} = \overline{\langle c^* | \phi_M \rangle} = \langle c | (\phi_M)^* \rangle;$$
on the other hand $\gamma(c) \in C(M)$ so 
$$\langle c | \gamma (\phi)_M \rangle = \gamma (\phi) (c) = \phi (\gamma(c)) = \langle \gamma(c) | \phi_M \rangle = \langle c | \gamma (\phi_M) \rangle.$$
\end{proof} 

\medbreak 

Similarly direct calculations show that $\tau (\phi) = \tau (\phi_M)$ and that if $a \in C(M)$ then $(\Theta_a)_M = a$. 

\medbreak 

\section*{A Partial Product}

\medbreak 

The considerations with which we closed the previous section lead us to consider the following definition of a partial product in the antidual. 

\medbreak 

Let $\phi, \psi \in C'(V)$. If $c \in C(V)$ is arbitrary then we define $(\phi \psi) (c)$ by the rule 
$$(\phi \psi) (c) = \lim_{M \uparrow \FF} \psi (\phi_M^* \: c)$$
or 
$$(\phi \psi) (c) = \lim_{M \uparrow \FF} \phi (c \: \psi_M^*)$$
{\it provided these limits exist}. Notice that if one of these limits exists then so does the other and the two coincide: in fact, as soon as $M \in \FF$ contains $M_c$ (equivalently, $C(M)$ contains $c$) it follows that 
$$\psi (\phi_M^* \: c) = \langle \phi_M^* \: c | \psi_M \rangle = \langle c | \phi_M  \psi_M \rangle = \langle c \: \psi_M^* | \phi_M \rangle = \phi (c \: \psi_M^*);$$
consequently, it is also the case that 
$$(\phi \psi) (c) = \lim_{M \uparrow \FF} \langle c | \phi_M  \psi_M \rangle.$$

\medbreak 

We may view this rule as defining a partial antifunctional $\phi \psi$ on $C(V)$: the set of all $c$ in $C(V)$ for which the displayed limits exist is plainly a subspace of $C(V)$ on which $\phi \psi$ is linear. However, we prefer to regard this rule as defining a partial product in $C'(V)$: thus, we define $\phi \psi : C(V) \ra \mathbb{C}$ by requiring that 
$$(\phi \psi) (c) = \lim_{M \uparrow \FF} \psi (\phi_M^* \: c) = \lim_{M \uparrow \FF} \phi (c \: \psi_M^*)$$
provided that these limits exist for each $c \in C(V)$. With this understanding, it is evident that $\phi \psi$ is an antilinear functional on $C(V)$; so $\phi \psi \in C'(V)$ when defined. We shall find it convenient to express the statement that the product $\phi \psi$ is defined in either of the equivalent forms $\psi \in D (\phi , -)$ and $\phi \in D( - , \psi)$. The subsets $D (\phi , -)$ and $D( - , \psi)$ of $C'(V)$ so defined are of course subspaces. 

\medbreak 

The linear embedding $\Theta : C(V) \ra C'(V)$ is multiplicative. 

\medbreak 

\begin{theorem} \label{mult}
If $a, b \in C(V)$ then $\Theta_a \Theta_b$ is defined and equals $\Theta_{a b}$. 
\end{theorem} 

\begin{proof} 
Let $c \in C(V)$ be arbitrary: if $M \in \FF$ contains $a$ then Theorem \ref{prop} justifies 
$$\Theta_b ((\Theta_a)_M^* \: c) = \Theta_b (a^* c) = \langle a^* c | b \rangle = \langle c | a b \rangle$$ 
so that 
$$\lim_{M \uparrow \FF} \Theta_b ((\Theta_a)_M^* \: c) =  \langle c | a b \rangle = \Theta_{a b} (c).$$ 
\end{proof} 

\medbreak 

The partial product relates correctly to the involution. 

\medbreak 

\begin{theorem} \label{invol}
Let $\phi, \psi \in C'(V)$. If $\psi \in D( \phi, - )$ then $\psi^* \in D( -, \phi^* )$ and $\psi^* \phi^* = (\phi \psi)^*$. 
\end{theorem} 

\begin{proof} 
Let $\psi \in D( -, \phi)$ and $c \in C(V)$. If $M \in \FF$ then Theorem \ref{prop} justifies 
$$\phi^* ((\psi^*)_M^* \: c) = \phi^* (\psi_M  c) = \overline{\phi ((\psi_M c)^*)} = \overline{\phi (c^* \psi_M^* )}.$$
Here, let $M \uparrow \FF$: it follows that $\phi (c^* \psi_M^* ) \ra (\phi \psi) (c^*)$ and therefore that $\phi^* ((\psi^*)_M^* \: c)$ converges to $\overline{ (\phi \psi ) (c^*)} = (\phi \psi )^* (c)$. 
\end{proof} 

\medbreak 

The partial product relates correctly to the grading. 

\medbreak 

\begin{theorem} \label{grad} 
Let $\phi, \psi \in C'(V)$. If $\psi \in D( \phi, - )$ then $\gamma (\psi) \in D( \gamma (\phi), - )$ and $\gamma(\phi) \gamma(\psi) = \gamma (\phi \psi)$. 
\end{theorem} 

\begin{proof} 
Let $\psi \in D( -, \phi)$ and $c \in C(V)$. As $M \uparrow \FF$  Theorem \ref{prop} justifies 
$$\gamma(\psi) (\gamma(\phi)_M^* \: c) = \gamma(\psi) (\gamma (\phi_M^*) \: c) = \psi (\phi_M^* \gamma(c)) \ra (\phi \psi) (\gamma (c)) = \gamma(\phi \psi) (c).$$
\end{proof} 

\medbreak 

Our partial product also captures the $C(V)$-bimodule structure on $C'(V)$. 

\medbreak 

\begin{theorem} \label{bimod}
If $a \in C(V)$ and $\psi \in C'(V)$ then $\Theta_a \in D( -, \psi)$ and $\Theta_a \psi = a \cdot \psi$. 
\end{theorem} 

\begin{proof} 
Let $c \in C(V)$ be arbitrary: the net $(\psi ((\Theta_a)_M^* \: c) : M \in \FF)$ is eventually constant, for as soon as $M \in \FF$ contains $a$ it follows by Theorem \ref{prop} that 
$$\psi ((\Theta_a)_M^* \: c) = \psi( a^* c) = (a \cdot \psi) (c).$$
\end{proof} 

\medbreak 

Of course, it is similarly true that the product $\phi \Theta_b$ is defined and equals $\phi \cdot b$. 

\medbreak 

The question to what extent our partial product is associative is a little more involved. We offer here two positive answers to this question; our first is as follows. 

\medbreak 

\begin{theorem} \label{submod}
If $\psi \in C'(V)$ then $D (-, \psi)$ is a left $C(V)$-submodule: if $\phi \in D(-, \psi)$ and $a \in C(V)$ then $a \cdot \phi \in D(-, \psi)$ and $(a \cdot \phi) \psi = a \cdot (\phi \psi)$. 
\end{theorem} 

\begin{proof} 
Note first that if $M \in \FF$ contains $a$ then $(a \cdot \phi)_M = a \phi_M$. Now let $c \in C(V)$ be arbitrary and calculate: as $M \uparrow \FF$ it follows that 
$$\psi ((a \cdot \phi)_M^* \: c) = \psi ((a \phi_M)^* \: c) = \psi (\phi_M^* a^* c) \ra (\phi \psi) (a^* c) = (a \cdot (\phi \psi)) (c).$$
\end{proof} 

\medbreak 

Similarly, $D( \phi, -)$ is a right $C(V)$-submodule and $\phi (\psi \cdot b) = (\phi \psi) \cdot b$. 

\medbreak 

Our second positive answer pertains to products in which an element of the Clifford algebra lies between two elements of the antidual. 

\medbreak 

\begin{theorem} \label{mid}
If $\phi, \psi \in C'(V)$ and $d \in C(V)$ then the conditions $\phi \cdot d \in D( - , \psi)$ and $d \cdot \psi \in D( \phi, -)$ are equivalent and imply $(\phi \cdot d) \psi = \phi (d \cdot \psi).$
\end{theorem} 

\begin{proof} 
By symmetry, we need only tackle one direction of the equivalence. Let $\phi \cdot d \in D( - , \psi)$: as in the proof of Theorem \ref{submod}, if $M \in \FF$ contains $d$ then $(\phi \cdot d)_M = \phi_M d$ so that if $c \in C(V)$ is arbitrary then passage to the limit as $M \uparrow \FF$ has the effect 
$$ (d \cdot \psi) (\phi_M^* c) = \psi( d^* \phi_M^* c) = \psi ((\phi \cdot d)_M^* c) \ra ((\phi \cdot d) \psi) (c).$$
\end{proof} 

\medbreak 

As announced prior to the statement of these two theorems, they may be recast as extensions of the associative law beyond the Clifford algebra and into the antidual: Theorem \ref{bimod} permits us to write the identities of Theorem \ref{submod} as $(\Theta_a \phi) \psi = \Theta_a ( \phi \psi)$ and $\phi (\psi \Theta_b) = (\phi \psi) \Theta_b$ and the identity of Theorem \ref{mid} as $(\phi \Theta_d) \psi = \phi (\Theta_d \psi)$; similarly, the $C(V)$-bimodule identity itself may be rewritten as $\Theta_a (\chi \Theta_b) = (\theta_a \chi) \Theta_b$. This leaves open the extent to which $(\phi \chi) \psi = \phi (\chi \psi)$ is satisfied in general. 

\medbreak 

The partial product in $C'(V)$ is unital: in fact, the antitrace 
$$\bar{\tau} : C(V) \ra \mathbb{C}: c \mapsto \overline {\tau (c)} = \langle c | {\bf 1} \rangle $$
serves as multiplicative identity. 

\medbreak 

\begin{theorem} \label{id} 
If $\chi \in C'(V)$ then $\bar{\tau} \chi = \chi = \chi \bar{\tau}$. 
\end{theorem} 

\begin{proof} 
If $c \in C(V)$ and $M \in \FF$ then 
$$\bar{\tau} (\chi_M^* c) = \langle \chi_M^* c | {\bf 1} \rangle = \langle c | \chi_M \rangle = \chi (c)$$
as soon as $M$ contains $c$. Passage to the limit as $M \uparrow \FF$ proves that $\chi \in D (-, \bar{\tau})$ and that $\chi \bar{\tau} = \chi$; symmetry does the rest.  
\end{proof} 

\medbreak 

Notice that $\bar{\tau} = \Theta_{{\bf 1}}$; in this sense, the multiplicative identity ${\bf 1}$ of the Clifford algebra continues to serve as a multiplicative identity for the antidual. 

\medbreak 

While on the topic of the trace, we may now return to the question that prompted our consideration of the partial product at the close of the previous section. 

\medbreak 

\begin{theorem} \label{tr}
Let $\phi, \psi \in C'(V)$. If $\phi \psi$ and $\psi \phi$ are both defined then $\tau (\phi \psi) = \tau (\psi \phi)$. 
\end{theorem} 

\begin{proof} 
Automatic: if $M \in \FF$ then 
$$\psi (\phi_M^*) = \langle \phi_M^* | \psi_M \rangle = \langle \psi_M^* | \phi_M \rangle = \phi (\psi_M^*)$$
whence passage to the limit as $M \uparrow \FF$ yields 
$$\tau (\phi \psi) = (\phi \psi) ({\bf 1}) = (\psi \phi) ({\bf 1}) = \tau (\psi \phi).$$
\end{proof} 

\medbreak 

Incidentally, with reference to the alternative view that regards the definition of the partial product as instead defining products to be partial functions, this proof shows that ${\bf 1}$ lies in the domain of the partial function $\phi \psi$ if and only if ${\bf 1}$ lies in the domain of the partial function $\psi \phi$ (with equal values). 

\medbreak 

\section*{The Clifford Hilbert Space}

\medbreak 

Thus far, we have only seen that if $\phi, \psi \in C'(V)$ then the product $\phi \psi$ is defined when at least one of its factors lies in the Clifford algebra (more precisely, in the image of its natural embedding in the antidual). Here, we show that the product is also defined when each of the factors is a {\it bounded} antifunctional; this leads to a canonical construction of the Hilbert space completion of $C(V)$ relative to its tracial inner product. 

\medbreak 

We begin by observing that the operator norm of an antifunctional may be recovered from the Clifford algebra norms of its finite-dimensional approximants. 

\medbreak 

\begin{theorem} \label{norm}
If $\phi \in C'(V)$ then 
$$|| \phi || = \sup \{ || \phi_M || : M \in \FF \} \in [0, \infty ].$$
\end{theorem} 

\begin{proof} 
Let $s$ be the indicated supremum: if $c \in C(V)$ then $c \in C(M)$ for some $M \in \FF$ and therefore 
$$| \phi (c) | = | \langle c | \phi_M \rangle | \leqslant || c || \: || \phi_M || \leqslant s || c ||;$$
this proves that $|| \phi || \leqslant s$. In the opposite direction, if $M \in \FF$ then $\phi_M \in C(M)$ so 
$$|| \phi_M ||^2 = \langle \phi_M | \phi_M \rangle = \phi (\phi_M) \leqslant || \phi || \: || \phi_M ||;$$
whether zero or not, $|| \phi_M ||$ is at most $|| \phi ||$. 
\end{proof} 

\medbreak 

Let us agree to write $C[V] \subset C'(V)$ for the space of all {\it bounded} antifunctionals: equivalently, those for which the supremum here is {\it finite}. 

\medbreak 

Let $M \in \FF$. As $C(M)$ is also finite-dimensional, its orthogonal space $C(M)^{\perp} \subset C(V)$ is complementary, so that 
$$C(V) = C(M)^{\perp} \oplus C(M).$$
We shall write $P_M$ for the corresponding orthogonal projector from $C(V)$ on $C(M)$ along $C(M)^{\perp}$. 

\medbreak 

Now let $\phi \in C'(V)$. We claim that if $N \in \FF$ contains $M \in \FF$ then 
$$P_M (\phi_N) = \phi_M.$$
In fact, if $c \in C(M)$ then $c \in C(N)$ so 
$$\langle c | P_M (\phi_N) \rangle = \langle c | \phi_N \rangle = \phi (c) = \langle c | \phi_M \rangle.$$
It follows that $\phi_N \in C(N)$ decomposes orthogonally as 
$$\phi_N = (\phi_N - \phi_M) + \phi_M \in C(M)^{\perp} \oplus C(M)$$
and that 
$$|| \phi_N ||^2 = || \phi_N - \phi_M ||^2 + || \phi_M ||^2.$$
In particular, it follows that the net $(|| \phi_M || : M \in \FF )$ is increasing. 

\medbreak 

We are now able to improve Theorem \ref{net} for bounded antifunctionals. 

\medbreak 

\begin{theorem} \label{normnet}
If $\phi \in C[V]$ then the net $( \Theta_{\phi_M} : M \in \FF )$ is norm convergent to $\phi$. 
\end{theorem} 

\begin{proof} 
Let $N \in \FF$ contain $M \in \FF$. Note that 
$$(\phi - \Theta_{\phi_M})_N = \phi_N - \phi_M$$
for if $c \in C(N)$ then 
$$\langle c | (\phi - \Theta_{\phi_M})_N \rangle = (\phi - \Theta_{\phi_M}) (c) = \phi (c) - \langle c | \phi_M \rangle = \langle c | \phi_N - \phi_M \rangle.$$
Accordingly, the Pythagorean identity displayed above may be written as 
$$|| \phi_N ||^2 = ||(\phi - \Theta_{\phi_M})_N ||^2 + ||\phi_M||^2$$
whence passage to the supremum as $N$ runs over $\FF$ yields  
$$|| \phi ||^2 = || \phi - \Theta_{\phi_M} ||^2 + || \phi_M ||^2$$
by Theorem \ref{norm}. Finally, as $M \uparrow \FF$ so $|| \phi_M || \uparrow || \phi ||$ and in conclusion $|| \phi - \Theta_{\phi_M} || \ra 0$. 
\end{proof} 

\medbreak 

We now see that $C[V]$ is a canonical model for the Hilbert space completion of $C(V)$ as an inner product space: on the one hand, if $a \in C(V)$ then $|| \Theta_a || = || a ||$ so the canonical linear embedding $\Theta$ defines an isometry $C(V) \ra C[V]$; on the other hand, Theorem \ref{normnet} implies that the image of $C(V)$ under $\Theta$ is dense in $C[V]$.  The inner product in $C[V]$ is of course given by the familiar rule
$$\langle \phi | \psi \rangle = \sum_{n = 0}^{3} i^{-n} || \phi + i^n \psi ||^2.$$

\medbreak 

The subspace $C[V] \subset C'(V)$ is clearly invariant under both involution and the grading automorphism: indeed, if $\phi \in C[V]$ then $|| \phi^* || = || \phi ||$ and $|| \gamma (\phi) || = || \phi ||$. It is also a sub $C(V)$-bimodule: in the following one-sided statement of this fact, $||| a |||$ denotes the operator norm of multiplication by $a$ in $C(V)$ on either left or right; this is also the (finite) norm of $a$ as an element of the Clifford C${^*}$-algebra as in Section 1.2 of [1]. 

\medbreak 

\begin{theorem} \label{subbi}
If $a \in C(V)$ and $\phi \in C[V]$ then $a \cdot \phi \in C[V]$ and $|| a \cdot \phi || \leqslant ||| a ||| \: || \phi ||.$
\end{theorem} 

\begin{proof} 
If $c \in C(V)$ then 
$$| (a \cdot \phi) (c)| = | \phi (a^* c)| \leqslant || \phi || \: || a^* c|| \leqslant || \phi || \: ||| a^* ||| \: ||c|| = ||| a ||| \: || \phi || \: ||c||$$
so that $a \cdot \phi$ is indeed bounded and has operator norm at most equal to $||| a ||| \: || \phi ||$. 
\end{proof} 

\medbreak

We are now in a position to prove that the partial product is defined for any pair of bounded antifunctionals: that is, if $\chi \in C[V]$ then $C[V] \subseteq D(\chi, -) \cap D(-, \chi)$. 

\medbreak 

\begin{theorem} \label{[][]}
If $\phi, \psi \in C[V]$ then the product $\phi \psi \in C'(V)$ is defined: explicitly, if $c \in C(V)$ then $$(\phi \psi) (c) = \langle \phi^* \cdot c | \psi \rangle = \langle c \cdot \psi^* | \phi \rangle.$$ 
\end{theorem} 

\begin{proof} 
If $M \in \FF$ contains $M_c$ then 
$$\psi (\phi_M^* c) = \langle \phi_M^* c | \psi_M \rangle = \langle (\phi^* \cdot c)_M| \psi_M \rangle = \langle \Theta_{(\phi^* \cdot c)_M} | \Theta_{\psi_M} \rangle$$
because $\Theta$ is isometric. Pass to the $M \uparrow \FF$ limit:  bearing Theorem \ref{subbi} in mind, Theorem \ref{normnet} implies $\Theta_{(\phi^* \cdot c)_M} \ra \phi^* \cdot c$ and $\Theta_{\psi_M} \ra \psi$ in norm, whence $\psi (\phi_M^* c) \ra \langle \phi^* \cdot c | \psi \rangle$ as required. 
\end{proof} 

\medbreak

Notice we do not assert here that the product $\phi \psi \in C'(V)$ invariably lies in $C[V]$. 

\medbreak 

In particular, it now follows that if $\phi, \psi \in C[V]$ then 
$$\langle \phi | \psi \rangle = (\phi^* \psi ) ({\bf 1}) = \tau (\phi^* \psi).$$
Thus, the formula by which $\langle \bullet | \bullet \rangle$ is defined on $C(V)$ continues to hold on $C[V]$. 

\medbreak 

Incidentally, it is also now clear that our partial product really is partial: for example, if $\phi \in C'(V)$ is {\it unbounded} then ${\bf 1}$ does not lie in the domain of the partial functional $\phi^* \phi$ so that $\phi^* \notin D (-, \phi)$ and the product $\phi^* \phi$ is not defined as a functional $C(V) \ra \mathbb{C}$. 

\bigbreak 

\begin{center} 
{\small R}{\footnotesize EFERENCES}
\end{center} 
\bigbreak 

[1] R.J. Plymen and P.L. Robinson, {\it Spinors in Hilbert Space}, Cambridge Tracts in Mathematics {\bf 114} (1994). 

\medbreak 

[2] P.L. Robinson, {\it `Twisted Duality' for Clifford Algebras}, arXiv 1407.1420 (2014). 

\medbreak

\end{document}